\documentclass[12pt]{article}

\usepackage{amsfonts,amssymb,amsmath,amsthm,epsfig}

\usepackage{graphics,graphicx}                 

\newtheorem{theorem}{Theorem}

\newtheorem{lemma}[theorem]{Lemma}

\theoremstyle{definition}
{

}

\begin{document}

\markboth{Aistis Atminas, Sergey Kitaev, Vadim V. Lozin, and Alexandr Valyuzhenich}
{Universal graphs and universal permutations}

\title{Universal graphs and universal permutations}

\author{Aistis Atminas\thanks{DIMAP and Mathematic Institute, University of Warwick, Coventry CV4 7AL, UK. Email: A.Atminas@warwick.ac.uk} \and 
Sergey Kitaev\thanks{Department of Computer and Information Sciences, University of Strathclyde, Glasgow G1 1XH, UK. Email: sergey.kitaev@cis.strath.ac.uk} \and
Vadim V. Lozin\thanks{DIMAP and Mathematic Institute, University of Warwick, Coventry CV4 7AL, UK. Email: V.Lozin@warwick.ac.uk}
\and
Alexandr Valyuzhenich\thanks{Novosibirsk State University, 2 Pirogova Street, 630090 Novosibirsk, Russia. Email: graphkiper@mail.ru}
}

\date{}

\maketitle

\begin{abstract}
Let $X$ be a family of graphs and $X_n$ the set of $n$-vertex graphs in $X$. 
A graph $U^{(n)}$ containing all graphs from $X_n$ as induced subgraphs is called $n$-universal for $X$. 
Moreover, we say that $U^{(n)}$ is a {\it proper} $n$-universal graph for $X$ if it belongs to $X$. 
In the present paper, we construct a proper $n$-universal graph for the class of split permutation graphs.
Our solution includes two ingredients: a proper universal $321$-avoiding permutation and
a bijection between $321$-avoiding permutations and symmetric split permutation graphs.
The $n$-universal split permutation graph constructed in this paper has $4n^3$ vertices, 
which means that this construction is order-optimal.  \\

\noindent
{\bf Keywords:} universal graphs; bipartite permutation graphs; split permutation graphs; $321$-avoiding permutations.
\end{abstract}

\section{Introduction}
Let $X$ be a family of graphs and $X_n$ the set of $n$-vertex graphs in $X$. 
A graph containing all graphs from $X_n$ as induced subgraphs is called $n$-universal for $X$.
The problem of constructing universal graphs is closely related to graph representations and finds applications in theoretical computer
science \cite{Alstrup,implicit}. This problem is trivial if universality is the only requirement, 
since the union of all vertex disjoint graphs from $X_n$ is 
obviously $n$-universal for $X$. However, this construction is generally neither optimal, in terms of the number 
of its vertices, nor proper, in the sense that it does not necessarily belong to $X$. 

Let us denote an $n$-universal graph for $X$ by $U^{(n)}$ and the set of its vertices by $V(U^{(n)})$.
Since the number of $n$-vertex subsets of $V(U^{(n)})$ cannot be smaller than the number of graphs in $X_n$, 
we conclude that 
$$
\log_2 |X_n|\le \log_2\binom{|V(U^{(n)})|}{n}\le n\log_2|V(U^{(n)})|.
$$
Also, trivially, $n\le |V(U^{(n)})|$, and hence, 
$$
n\log_2 n\le n\log_2|V(U^{(n)})|.
$$
We say that $U^{(n)}$ is {\it optimal} if $n\log_2|V(U^{(n)})|=\max(\log_2 |X_n|,n\log_2 n)$, {\it asymptotically optimal} if 
$$
\lim_{n\to \infty}\frac{n\log_2|V(U^{(n)})|}{\max(\log_2 |X_n|,n\log_2 n)}=1,
$$ 
and {\it order-optimal} if there is a constant $c$ such that for all $n\ge 1$,
$$
\frac{n\log_2|V(U^{(n)})|}{\max(\log_2 |X_n|,n\log_2 n)}\le c.
$$ 
Optimal universal graphs (of various degrees of optimality) have been constructed
for many graph classes such as the class of all graphs \cite{Moon}, threshold graphs \cite{threshold},
planar graphs \cite{Chung}, graphs of bounded arboricity \cite{Alstrup}, 
of bounded vertex degree \cite{Butler,ELO08}, split graphs and bipartite graphs \cite{Lozin}, 
bipartite permutation graphs \cite{Gabor}, etc. Some of these constructions can also 
be extended to an infinite universal graph, in which case the question 
of optimality is not relevant any more and the main problem is finding a universal element 
{\it within} the class under consideration. 
We call a universal graph for a class $X$ that belongs to $X$ a {\it proper} universal graph. 
  
Not for every class there exist proper infinite universal graphs with countably many vertices \cite{FK97}.
Such constructions are known for the class of all graphs \cite{Rado}, $K_4$-free graphs and some other classes \cite{uni1,uni2}.
A proper infinite countable universal graph can be also easily constructed for the class of bipartite permutation 
graphs from finite $n$-universal graphs (represented in Figure~\ref{fig1}) by increasing $n$ to infinity.

In the present paper, we study a related class, namely, the class of split permutation graphs.
In spite of its close relationship to bipartite permutation graphs, finding a {\it proper}
$n$-universal graph for this class is not a simple problem even for finite values of $n$.  
We solve this problem  by establishing a bijection between symmetric split permutation graphs 
and $321$-avoiding permutations and by constructing a proper universal $321$-avoiding permutation. 
Our construction uses $4n^3$ vertices. Since there are at most $n!< n^n$ labeled permutation graphs, 
this construction is order-optimal. Whether this construction can be extended to an infinite countable
graph remains a challenging open problem.

The organization of the paper is as follows. In Section~\ref{sec:pre} we introduce all
preliminary information related to the topic of the paper. Then in Section~\ref{sec:per}
we construct a universal $321$-avoiding permutation and in Section~\ref{sec:spg}
we use this construction to build a universal split permutation graph. Finally, 
in Section~\ref{sec:con} we conclude the paper with a number of open problems.

\section{Preliminaries}
\label{sec:pre}
All graphs in this paper are finite, undirected, and without loops or multiple edges. 
We denote the set  of vertices of a graph $G$ by $V(G)$ and the set of its edges by $E(G)$.
Given a vertex $v\in V(G)$, we denote by $N(v)$ the neighbourhood of $v$,
i.e. the set of vertices adjacent to $v$. For a subset $U\subseteq V(G)$, we let
$G[U]$ denote the subgraph of $G$ induced by $U$, i.e. the vertex set of $G[U]$ is $U$ 
with two vertices being adjacent in $G[U]$ if and only if they are adjacent in $G$.

We say that a graph $G$ contains  a graph $H$ as an induced subgraph if $H$ is isomorphic to an induced subgraph of $G$. 
A {\it clique} in a graph is a subset of pairwise adjacent vertices and an {\it independent set} is 
a subset of pairwise non-adjacent vertices.

Three families of graphs are of special interest in this paper. These are bipartite graphs,
split graphs and permutation graphs. A graph $G$ is {\it bipartite} if its vertices can be partitioned into at most two independent
sets, and $G$ is a {\it split graph} if its vertices can be partitioned into an independent set and a clique.
To define the notion of a permutation graph, let us first introduce some terminology related to permutations.


\medskip
A permutation on the set $[n]:=\{1,2,\ldots,n\}$ is a bijection from the set to itself.
A commonly used way of representing a permutation $\pi :[n]\to [n]$ is the one-line notation, 
which is the ordered sequence $\pi(1)\pi(2)\cdots\pi(n)$. A permutation $\pi$ is said to {\it contain} 
a permutation $\rho$ as a pattern if $\pi$ has a subsequence that is order isomorphic to $\rho$. If $\pi$ does not 
contain $\rho$, we say that $\pi$ {\it avoids} $\rho$. 

To ``visualize'' the pattern containment relation, we represent permutations by means of intersection diagrams, 
as illustrated in Figure~\ref{fig:pi}. Clearly, such a diagram uniquely describes the permutation even 
without the labels attached to the endpoints of the line segments. We use the labels only for convenience.
Then a permutation $\pi$ contains a permutation $\rho$ if the (unlabeled) diagram representing $\rho$
can be obtained from the (unlabeled) diagram representing $\pi$ by deleting some  segments.  
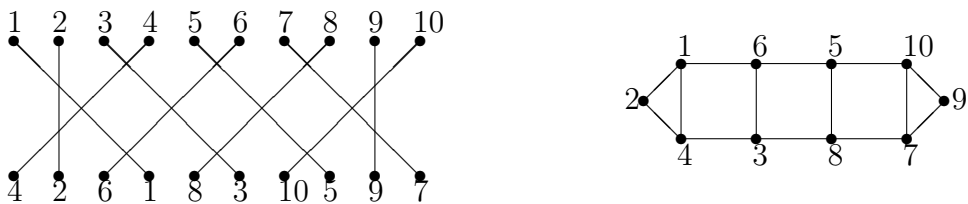
\begin{figure}[ht]
\begin{center}
\begin{picture}(220,90)
\setlength{\unitlength}{0.3mm}
\put(0,0){\circle*{5}} 
\put(20,0){\circle*{5}}
\put(40,0){\circle*{5}} 
\put(60,0){\circle*{5}}
\put(80,0){\circle*{5}} 
\put(100,0){\circle*{5}}
\put(120,0){\circle*{5}} 
\put(140,0){\circle*{5}}
\put(160,0){\circle*{5}} 
\put(180,0){\circle*{5}}

\put(0,60){\circle*{5}} 
\put(20,60){\circle*{5}}
\put(40,60){\circle*{5}} 
\put(60,60){\circle*{5}}
\put(80,60){\circle*{5}} 
\put(100,60){\circle*{5}}
\put(120,60){\circle*{5}} 
\put(140,60){\circle*{5}}
\put(160,60){\circle*{5}} 
\put(180,60){\circle*{5}}

\put(0,0){\line(1,1){60}} 
\put(20,0){\line(0,1){60}}
\put(40,0){\line(1,1){60}} 
\put(60,0){\line(-1,1){60}}
\put(80,0){\line(1,1){60}} 
\put(100,0){\line(-1,1){60}}
\put(120,0){\line(1,1){60}} 
\put(140,0){\line(-1,1){60}}
\put(160,0){\line(0,1){60}} 
\put(180,0){\line(-1,1){60}} 
\put(-3,-11){$4$} 
\put(17,-11){$2$}
\put(-3,64){$1$}
\put(17,64){$2$}
\put(37,-11){$6$} 
\put(57,-11){$1$}
\put(37,64){$3$}
\put(57,64){$4$}
\put(77,-11){$8$} 
\put(97,-11){$3$}
\put(77,64){$5$}
\put(97,64){$6$}

\put(117,-11){$10$} 
\put(137,-11){$5$}
\put(117,64){$7$}
\put(137,64){$8$}
\put(157,-11){$9$} 
\put(177,-11){$7$}
\put(157,64){$9$}
\put(177,64){$10$}
\end{picture}
\begin{picture}(130,70)
\setlength{\unitlength}{0.5mm}
\put(10,20){\circle*{3}} 
\put(20,10){\circle*{3}}
\put(40,10){\circle*{3}} 
\put(60,10){\circle*{3}}
\put(80,10){\circle*{3}} 
\put(90,20){\circle*{3}} 
\put(20,30){\circle*{3}}
\put(40,30){\circle*{3}} 
\put(60,30){\circle*{3}}
\put(80,30){\circle*{3}} 

\put(20,30){\line(1,0){20}} 
\put(20,30){\line(0,-1){20}}
\put(20,30){\line(-1,-1){10}} 
\put(20,10){\line(-1,1){10}}
\put(40,10){\line(1,0){20}} 
\put(40,10){\line(-1,0){20}}
\put(40,10){\line(0,1){20}} 
\put(60,30){\line(-1,0){20}}
\put(60,30){\line(1,0){20}}
\put(60,30){\line(0,-1){20}} 
\put(80,10){\line(-1,0){20}} 
\put(80,10){\line(1,1){10}}
\put(80,10){\line(0,1){20}} 
\put(80,30){\line(1,-1){10}}

\put(19,3){$4$} 
\put(19,32){$1$}
\put(5,18){$2$}
\put(39,3){$3$}
\put(39,32){$6$} 
\put(59,3){$8$}
\put(59,32){$5$}
\put(79,3){$7$}
\put(79,32){$10$} 
\put(92,18){$9$}
\end{picture}
\end{center}
\caption{The diagram representing the permutation $\pi=426183(10)597$ (left) and the permutation graph $G_{\pi}$ (right).} 
\label{fig:pi}
\end{figure}

One more way to better understand the pattern containment relation on permutations is through the notion
of a permutation graph. The permutation graph $G_{\pi}$ of a permutation $\pi$ is the intersection graph of 
the line segments in the diagram representing $\pi$, i.e. the graph whose vertices are the segments with
two vertices being adjacent if and only if the respective segments intersect (cross) each other 
(again, see Figure~\ref{fig:pi} for an illustration). Then a permutation $\pi$ contains a permutation
$\rho$ as a pattern if and only if $G_{\pi}$ contains $G_{\rho}$ as an induced subgraph.

A graph $G$ is said to be a {\it permutation graph} if $G$ is isomorphic to $G_{\pi}$ for some permutation $\pi$.
In this paper, we focus on two particular subclasses of permutation graphs: bipartite permutation and split permutation 
graphs. Bipartite permutation graphs are precisely permutation graphs without a triangle (a clique of size 3), because,
first, bipartite graphs are precisely graphs without odd cycles, and second, no cycle with at least 5 vertices 
is a permutation graph (which can be easily seen). Since a triangle is the permutation graph of the permutation 
$321$, the class of bipartite permutation graphs consists of permutation graphs of $321$-avoiding permutations.
We study these permutations in the next section.

\section{Universal 321-avoiding permutations}
\label{sec:per}

In this section, we study the set of $321$-avoiding permutations, i.e. permutations
containing no $321$ as a pattern. In other words, a permutation is $321$-avoiding
if it contains no subsequence of length 3 in decreasing order reading from left to right,
or equivalently, if its elements can be partitioned into at most 2 increasing subsequences.
For example, the permutation $245136$ avoids the pattern $321$ (with 2456 and 13
being two increasing subsequences), while $261435$ does not because of the subsequence 643. 
It is known that the number of $321$-avoiding permutations of length $n$ is given by $C_n=\frac{1}{n+1}{2n\choose n}$, 
the {\em $n$-th Catalan number}. 
We let $S_n(321)$ denote the set of all $321$-avoiding permutations of length~$n$.  

The set of $321$-avoiding permutations have been studied regularly in the {\it theory of permutation patterns} 
(see \cite{K} for a recent comprehensive introduction to the respective field) in connection with various combinatorial problems. 
In this section, we study these permutations in connection with the notion of a {\it universal permutation}.

Given a set $X$ of permutations, we say that a permutation $\Pi$ is $n$-universal for $X$ if 
it contains all permutations of length $n$ from $X$ as patterns. Moreover, $\Pi$ is a {\it proper} $n$-universal permutation for $X$
if $\Pi$ belongs to $X$. 

Note that it is straightforward to construct a proper $n$-universal permutation for $S_n(321)$ of length 
$nC_n=\frac{n}{n+1}{2n \choose n}$. Indeed, we can list all the $C_n$ permutations in a row, say, 
in a lexicographic order, and create a single permutation by raising the elements of the $i$-th permutation 
from left by $(i-1)n$, for $1\leq i\leq C_n$; the resulting permutation will clearly avoid the pattern $321$. 
For example, for $n=3$, the elements of $S_3(321)$ can be listed as $123$, $132$, $213$, $231$ and $312$ leading to the permutation 
$$123465879(11)(12)(10)(15)(13)(14).$$ 
However, our goal in this section 
is to construct a proper $n$-universal permutation of length $n^2$ for the set of $321$-avoiding permutations of length $n$. To this end, 
we denote by $\rho_n$ the permutation that  begins with $(n+2)1(n+4)2\cdots (3n)n$ followed by 
$(3n+2)(n+1)(3n+4)(n+3)\cdots(5n)(3n-1)$ followed by $(5n+2)(3n+1)(5n+4)(3n+3)\cdots(7n)(5n-1)$ 
followed by $(7n+2)(5n+1)(7n+4)(5n+3)\cdots(9n)(7n-1)$, etc. In general, for $1<i<\lfloor \frac{n}{2}\rfloor$, the $i$-th $2n$-block of the permutation is given by 
$$(2ni-n+2)(2ni-3n+1)(2ni-n+4)(2ni-3n+3)\cdots(2ni+n)(2ni-n-1).$$ 
For $n\geq 3$, in case of even $n$, the last $2n$ elements of the permutation are
$$(n^2-n+1)(n^2-3n+1)(n^2-n+2)(n^2-3n+3)\cdots(n^2)(n^2-n-1)$$ 
while in case of odd $n$, the last $n$ elements of the permutation are $$(n^2-2n+1)(n^2-2n+3)\cdots(n^2-1).$$ 
For small values of $n$, we have the following:
$$
\begin{array}{l|l}
n & \rho_n \\
\hline
1 & 1 \\
\hline
2 &  3142\\
\hline
3 &  517293468\\
\hline
4 & 6182(10)3(12)4(13)5(14)7(15)9(16)(11)\\
\hline
5 & 7192(11)3(13)4(15)5(17)6(19)8(21)(10)(23)(12)(25)(14)(16)(18)(20)(22)(24)\\
\end{array}
$$

It is not difficult to see that, by construction, $\rho_n$ is a $321$-avoiding permutation. We denote the permutation graph of $\rho_n$ by $H_n$. From the intersection diagram of  $\rho_n$ shown in Figure \ref{fig2}, one can see that $H_n$ is the graph presented in Figure \ref{fig1}.

\begin{figure}[ht]
\begin{center}
\includegraphics[scale=0.5]{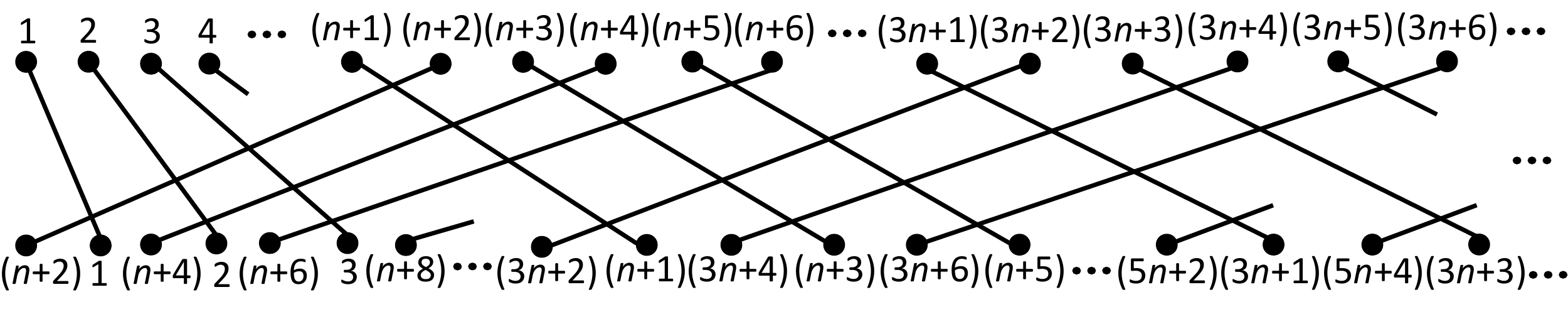}
\end{center}
\vspace{-0.5cm}
\caption{The intersection diagram of $\rho_n$.}\label{fig2} 
\end{figure}

If in a $321$-avoiding permutation of length $n$ the element $n-1$ is to the right of $n$, 
then $n-1$ must be the rightmost element in the permutation to avoid the pattern $321$ 
involving the elements $n-1$ and $n$. Using this observation, we can generate all $321$-avoiding permutations 
of length $n$ from $321$-avoiding permutations of length $n-1$ as follows. 
Let $\pi\in S_{n-1}(321)$. To obtain all permutations in $S_n(321)$ derived from $\pi$, 
we can insert $n$ in $\pi$ in any place to the right of $n-1$ (this cannot lead to an occurrence of the pattern $321$, 
say $nxy$, since $(n-1)xy$ would then be an occurrence of $321$ in $\pi$); 
also, we can replace the element $(n-1)$ in $\pi$ by $n$ and adjoin the element $(n-1)$ to the right of the obtained permutation. 
Clearly, if we apply the described operations to different permutations in $S_{n-1}(321)$, 
we will obtain different permutations in $S_n(321)$. Moreover,  the steps described above are reversible, 
namely for any permutation in $S_n(321)$ we can figure out from which permutation in $S_{n-1}(321)$ it was obtained, 
which gives us the desired.

\begin{figure}[ht]
\begin{center}
\includegraphics[scale=0.5]{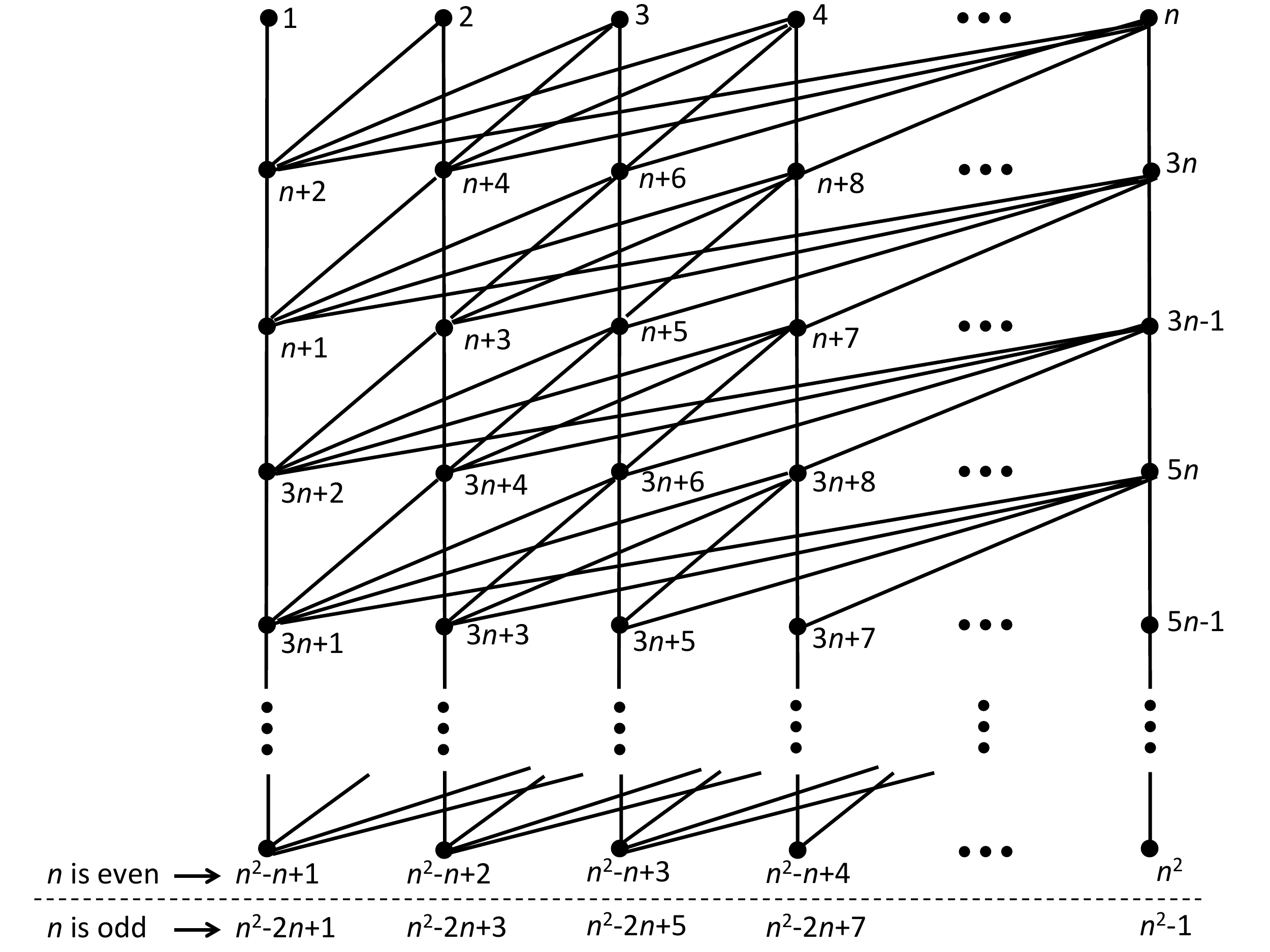}
\end{center}
\vspace{-0.5cm}
\caption{The graph $H_n$.}
\label{fig1} 
\end{figure}

The main result of the section is the following theorem.

\begin{theorem}\label{theorem} 
The permutation  $\rho_n$ is a proper $n$-universal permutation for the set $S_n(321)$. 
\end{theorem}

\begin{proof}  Note that if one removes the last column and last row in the graph in Figure \ref{fig1} 
and relabels vertices by letting a vertex receive label $i$ if it is the $i$-th largest label in the original labeling, 
then one gets exactly the graph in question of size $(n-1)\times(n-1)$. 
This observation allows us to apply induction on the size of the graph with obvious base case on 1 vertex 
giving a permutation containing the only 321-avoiding permutation of size 1.

\begin{figure}[ht]
\begin{center}
\includegraphics[scale=0.5]{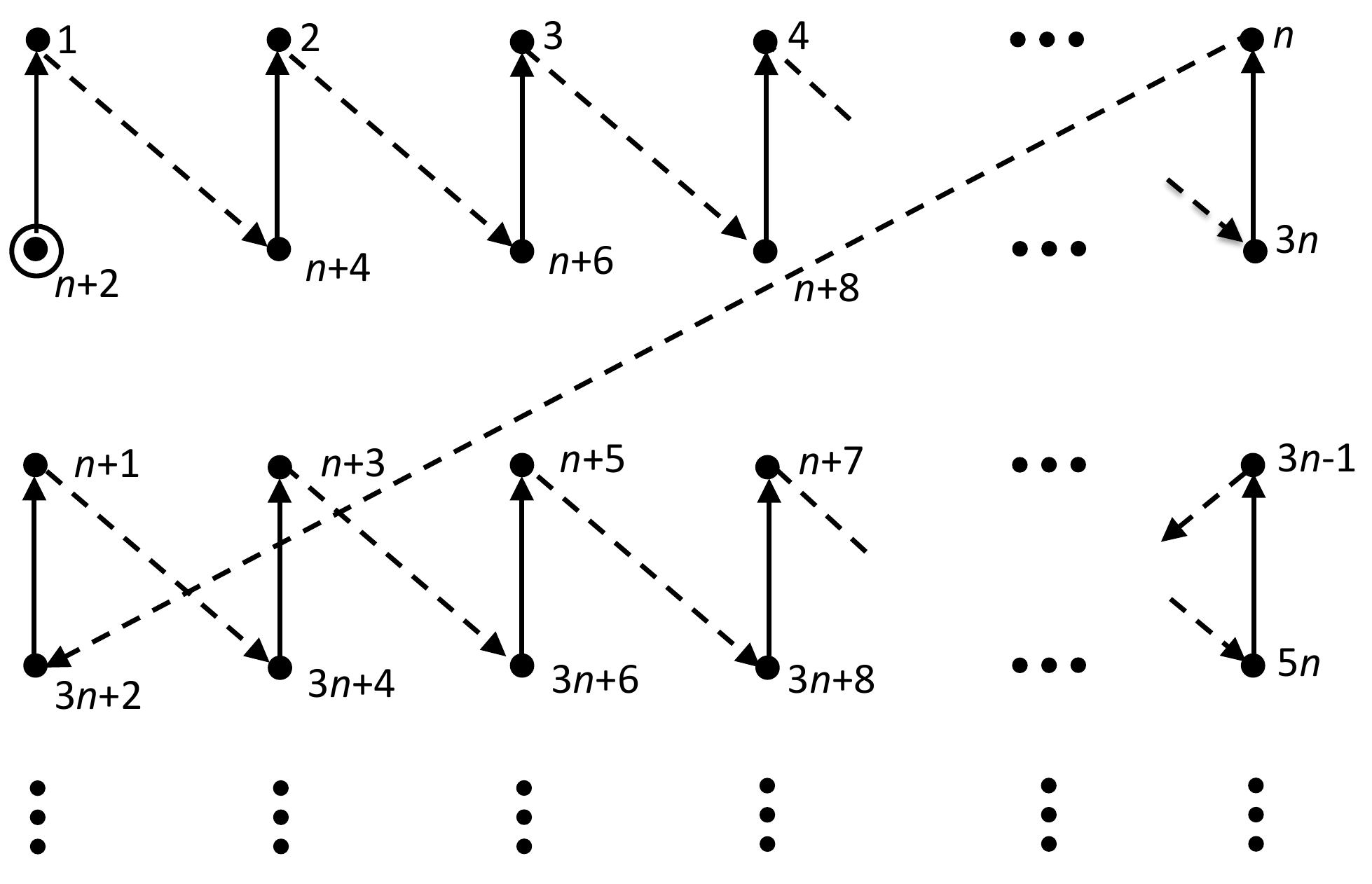}
\end{center}
\vspace{-0.5cm}
\caption{A direct reading off the permutation $\rho_n$ from the graph.}\label{fig3} 
\end{figure}

It is not hard to see based on Figure~\ref{fig2} that a direct way to read off 
the permutation corresponding to the graph in Figure~\ref{fig1} is to start with 
the element $n+2$ and to follow the (solid and dashed) arrows as shown in Figure~\ref{fig3}; 
in the case of odd $n$, the last row only has horizontal arrows (say, solid) going from left 
to right as shown schematically in the left picture in Figure \ref{fig4}. 
We only pay attention to the vertical solid arrows and introduce some terminology here. 
We say that an element is in a {\em top row} if an arrow points at it. 
By definition, all elements in the last row for the case when $n$ is odd are considered to be on a top row. 
Non-top elements are said to be from a {\em bottom row}. 
Thus, each element is either a top or a bottom element depending on which row it lies on. 
Two elements are said to be {\em neighbor elements} if they are connected by a vertical arrow. 
Clearly, for two neighbors  we always have that the bottom neighbor is larger than the top neighbor. 
Also, it is straightforward to see from the structure of the graph that on the same row, elements increase from left to right. 
Finally, a bottom element is always larger than the top element right below it.

\begin{figure}[ht]
\begin{center}
\includegraphics[scale=0.5]{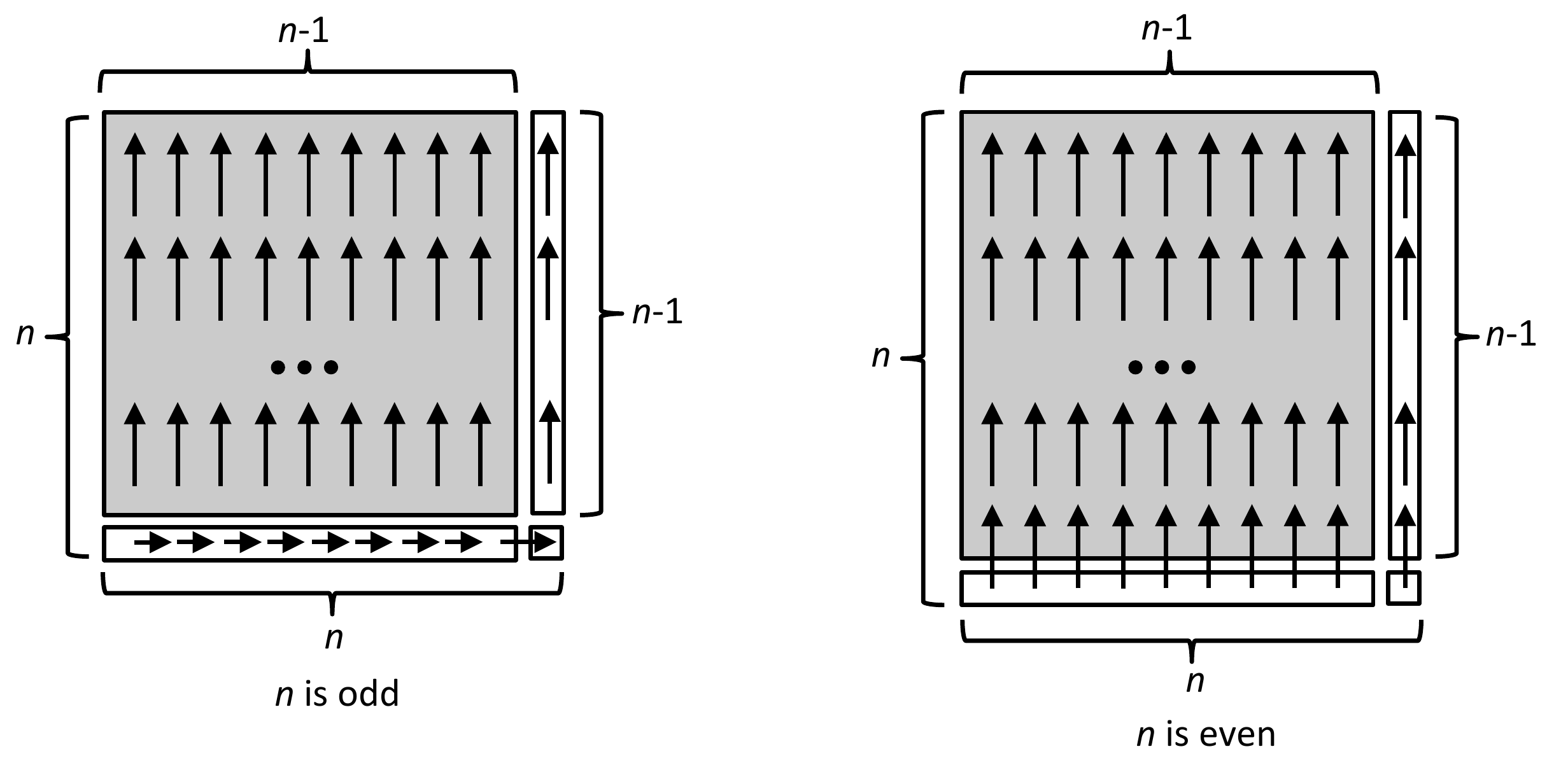}
\end{center}
\vspace{-0.5cm}
\caption{Applying inductive hypothesis.}\label{fig4} 
\end{figure}

Suppose that the statement is true for the case of $n-1$ and we would like to prove it for the case of $n$. 
That is, if in the graph on $n^2$ vertices we remove the last row and the last column, then in the obtained graph, 
shown schematically shaded in Figure \ref{fig4}, we can {\em realize} any $321$-avoiding permutation of length $n-1$. 
By realization of a permutation of length $n-1$ we mean picking $n-1$ elements in the shaded area in Figure \ref{fig4}, 
so that while going through the directed path in Figure \ref{fig3}, the picked elements form the same relative order as 
the elements of the permutation. Our goal is to show how to realize a $321$-avoiding permutation $\pi$ of length $n$ based on 
a realization of the permutation $\sigma$ obtained from $\pi$ either by removing the largest element (when $\pi$ does not end on $n-1$), or by removing $n-1$ from the end of $\pi$ and then replacing $n$ with $n-1$. 
Note that the element $n^2$ is not in the shaded area. We distinguish here between two cases. 

\medskip
{\it Case 1: $n$ is odd}. We split this case into two subcases as follows. 

\medskip
\noindent
{\it Case 1.1: $\pi$ does not end on $n-1$.} That is, in $\pi$, $n-1$ is to the left of $n$. 
If $\sigma$ ends with the largest element, we can always take the element $n^2$ to extend 
the realization of $\sigma$ to a realization of $\pi$. On the other hand, if $n-1$ is not 
the rightmost element in $\sigma$, then clearly the element corresponding to $n-1$, say $x$, 
must be in a bottom row, say, row $i$ ($i$ is even). 
Moreover, since the elements to the right of $n-1$ in $\sigma$ must be in increasing order 
to avoid the pattern $321$, and we have at most $n-2$ such elements, we know that those elements 
are chosen among the top elements in row $i-1$ weakly to the right of $x$ and, if $i+1< n-1$, 
possibly among the top elements in row $i+1$ weakly to the left of $x$. 
In either case, we can always extend the realization of $\sigma$ to a realization of $\pi$ 
by picking a bottom element to the right of $x$ in row $i$ or weakly to the left of $x$ in row $i+2$ (if it exists). 
Indeed, if $n$ is next to the left of an element $y$ in $\pi$, and vertex $z$ corresponds to $y$ 
in the realization of $\sigma$, then we can pick the (bottom) neighbor of $z$ to correspond to 
$n$ and thus to get a realization of $\pi$. For example, if $\pi=24513$ ($n=5$) and the realization of  
$\sigma=2413$ is the sequence 7, 11, 6, 10, that is, $x=11$, $y=1$ and $z=6$, then the neighbor of $z$ 
is $17$ and a desired realization of $\pi$ is 7, 11, 17, 6, 10; see the left picture in Figure~\ref{fig5}.

\begin{figure}[ht]
\begin{center}
\includegraphics[scale=0.5]{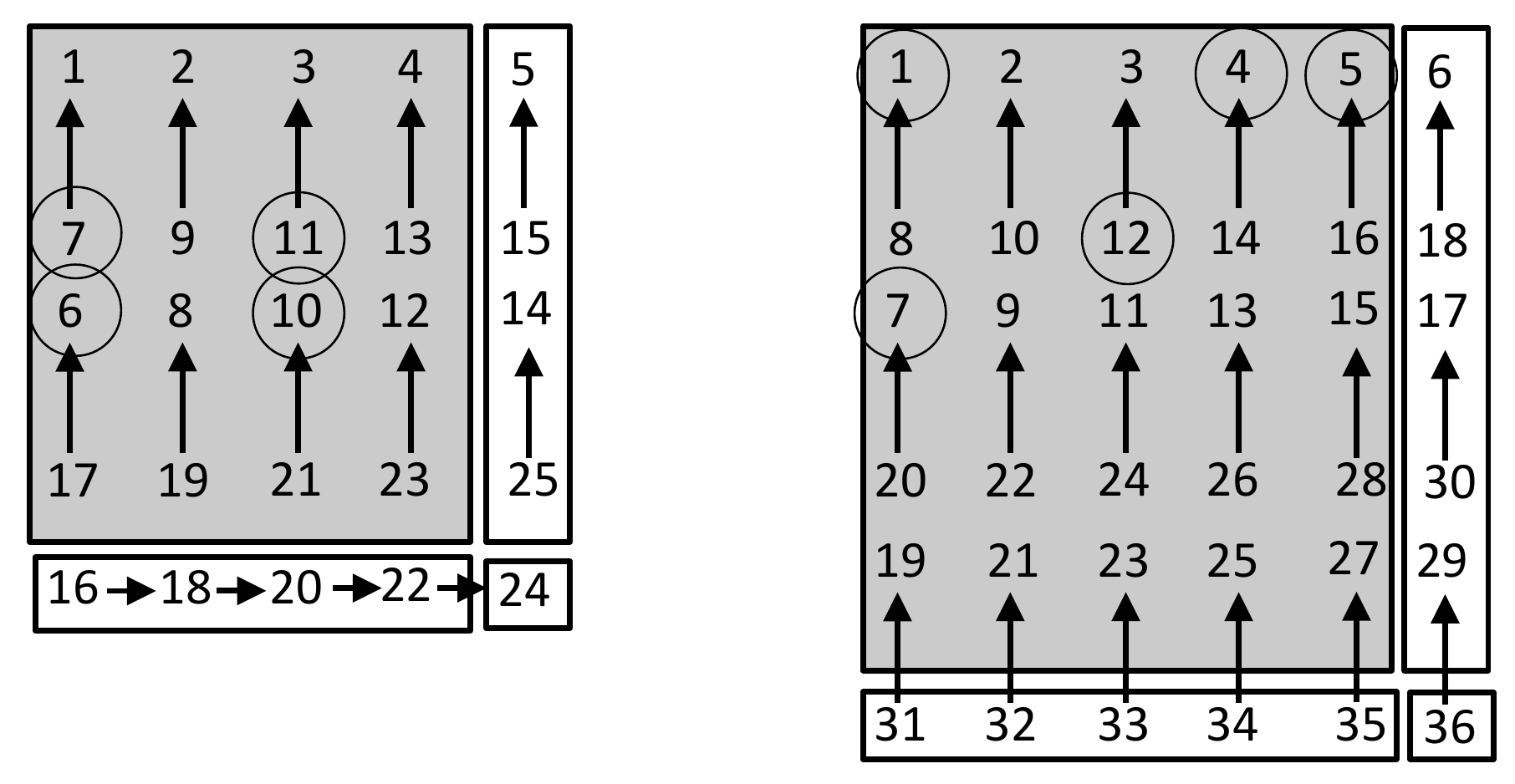}
\end{center}
\vspace{-0.5cm}
\caption{Examples illustrating our inductive proof.}\label{fig5} 
\end{figure}

The only case when the described approach does not work is if the bottom neighbor of $z$ is $x$ itself. In this case we will take a different realization of $\sigma$ obtained by one position shift to the right 
of all the elements already involved in the realization that are located weakly above or weakly to the right 
of $x$ including $x$ itself, and we pick   the element directly below the shifted $z$ to obtain the desired outcome.  See the respective subcase involving shifting in Case 1.2 below for more details on why the shift works and for an example of applying the shift.

\medskip
\noindent
{\it Case 1.2: $\pi$ ends on $n-1$.}  If the element $x$ corresponding to $n-1$ in a realization of $\sigma$ is in a top row, 
$\sigma$ must end with $n-1$ and we can add the (bottom) neighbor of $x$ to the realization of $\sigma$ to get a realization 
of $\pi$ (then the neighbor of $x$ will correspond to $n$, while the other elements in $\pi$ will have the same corresponding 
elements as those in $\sigma$).  On the other hand, if $x$ is in a bottom row, say $i$, then applying the same arguments as above, 
everything to the right of $n-1$ in $\sigma$ can be realized by the elements weakly to the right of $x$ in (top) row $i-1$ and 
weakly to the left of $x$ in (top) row $i+1$ if $i+1< n-1$.  

The worst case for us here is if the element directly below $x$ is involved in the realization of $\sigma$. 
In this case we will take a different realization of $\sigma$ obtained by one position shift to the right 
of all the elements already involved in the realization that are located weakly above or weakly to the right 
of $x$ including $x$ itself (that is, we shift  the realization elements in the first quadrant, including 
the positive semi-axes, with the origin at $x$). Such a shift is always possible because we have the $n$-th column adjoined. 
Also, such a shift does not change any order of elements involved in the initial realization of $\sigma$ 
because there are no elements of the realization in the fourth quadrant and the positive $x$-axis with the origin at $x$. 
For example, if $n=5$, $x=11$ and 10 and, say 3 and 4 would be in a realization of $\sigma$, after the shift, 
we would have 10, $x=13$, 4 and 5 in the realization of $\sigma$ instead of 10, 11, 3 and 4, respectively 
(see the left picture in Figure~\ref{fig5} for layout of elements when $n=5$; the circles there should be ignored for the moment). 

In either case, we can  pick the element directly below $x$ to play the role of $n-1$, 
while $x$ will be playing the role of $n$ in the obtained realization of $\pi$. 
For example, if $\pi=25134$ and  a realization of  $\sigma=2413$ is the sequence 7, 11, 6, 10 
as in our previous example ($x=11$ here), then we can consider instead the realization of $\sigma$ 
7, 13, 6, 10 and 7, 13, 6, 10, 12 is a desired realization of $\pi$; see the left picture in Figure~\ref{fig5}.

\medskip
{\it Case 2: $n$ is even}. 

\medskip
In this case, we actually can repeat verbatim all the arguments from the case ``$n$ is odd" keeping in mind that
\begin{itemize}
\item by definition, the last row is a top row when considering the odd $n-1$, and
\item picking the bottom neighbor element (in the case of even $n$) may involve the last row, which does not change anything.
\end{itemize} 

Of course, different examples should be provided in this case. 
For the subcase  ``$\pi$ does not end on $n-1$'', if $\pi=152634$ ($n=6$) and 
the realization of  $\sigma=15234$ is the sequence 1, 12, 4, 5, 7, that is, $x=12$, $y=3$ and $z=5$, 
then the neighbor of $z$ is $16$ and the desired realization of $\pi$ is 1, 12, 4, 16, 5, 7; 
see the right picture in Figure~\ref{fig5}. Finally, for the subcase ``$\pi$ ends on $n-1$'', 
if $\pi=162345$ and  a realization of  $\sigma=15234$ is the sequence 1, 12, 4, 5, 7 ($x=12$), 
then 1, 12, 4, 5, 7, 11 is a desired realization of $\pi$; see the right picture in Figure \ref{fig5}.
\end{proof}

By direct inspection, it is not difficult to verify that the permutation graph of $\rho_n$, $H_n$, 
is precisely the $n$-universal graph for the class of bipartite permutation graphs
described in \cite{Gabor}. Therefore, Theorem \ref{theorem} strengthens the main result of \cite{Gabor} by 
extending it from graphs to permutations (in particular, Theorem \ref{theorem} implies Theorem 6 in \cite{Gabor}).

\section{Universal split permutation graphs}
\label{sec:spg}

In this section, we use the result of Section~\ref{sec:per} to construct a universal split permutation graph.
To this end, let us first introduce some more terminology.

A {\it vicinal quasi-order} $\sqsubseteq$ on the vertex set of a graph is defined as: 
\[	x\sqsubseteq y\mbox{ if and only if }N(x)\subseteq N(y)\cup \{y\}.
\]

We call a set of vertices which are pairwise incomparable with respect to this relation a {\it vicinal antichain},
and a set of vertices which are pairwise comparable a {\it vicinal chain}.  

The {\it Dilworth number} of a graph is the maximum size of a vicinal antichain in
the graph, or equivalently, the minimum size of a partition of its vertices into
vicinal chains.

Split graphs of Dilworth number 1 are precisely {\it threshold graphs} \cite{CH77}.
In other words, a graph is threshold if and only if its vertices form a vicinal chain.
Let us emphasize that threshold graphs are split graphs, i.e. the vertices of a threshold graph
can be partitioned into a clique and an independent set. Moreover, it is not difficult to see that
\begin{itemize}
\item[(0)] if in a split graph with a clique $C$ and an independent set $I$ the vertices of $C$ or 
the vertices of $I$ form a vicinal chain, then the set of all vertices of the graph form a vicinal chain.
\end{itemize}

Universal threshold graphs have been constructed in \cite{threshold}. An $n$-universal
graph in this class contains $2n$ vertices of which $n$ vertices form a clique $C=(c_1,\ldots,c_n)$, 
$n$ vertices form an independent set $I=(i_1,\ldots,i_n)$ and for each $j$, $N(i_j)=\{c_1,\ldots,c_j\}$.
An example of an $n$-universal threshold graph is represented in Figure~\ref{fig:H5}

\begin{figure}[ht]
\begin{center}
\begin{picture}(300,60)
\put(50,0){\circle*{5}}
\put(100,0){\circle*{5}}
\put(150,0){\circle*{5}}
\put(200,0){\circle*{5}}
\put(250,0){\circle*{5}}
\put(48,53){$c_1$}
\put(98,53){$c_2$}
\put(148,53){$c_3$}
\put(198,53){$c_4$}
\put(248,53){$c_5$}


\put(150,52){\oval(230,20)}

\put(48,-12){$i_1$}
\put(98,-12){$i_2$}
\put(148,-12){$i_3$}
\put(198,-12){$i_4$}
\put(248,-12){$i_5$}
\put(50,50){\circle*{5}}
\put(100,50){\circle*{5}}
\put(150,50){\circle*{5}}
\put(200,50){\circle*{5}}
\put(250,50){\circle*{5}}

\put(50,0){\line(0,1){50}}
\put(250,0){\line(-1,1){50}}
\put(250,0){\line(-2,1){100}}
\put(250,0){\line(-3,1){150}}
\put(250,0){\line(-4,1){200}}
\put(100,0){\line(0,1){50}}
\put(200,0){\line(-1,1){50}}
\put(200,0){\line(-2,1){100}}
\put(200,0){\line(-3,1){150}}
\put(150,0){\line(0,1){50}}
\put(150,0){\line(-1,1){50}}
\put(150,0){\line(-2,1){100}}
\put(200,0){\line(0,1){50}}
\put(100,0){\line(-1,1){50}}
\put(250,0){\line(0,1){50}}

\end{picture}
\end{center}
\caption{Universal threshold graph for $n=5$ (the oval represents a clique).}
\label{fig:H5}
\end{figure}
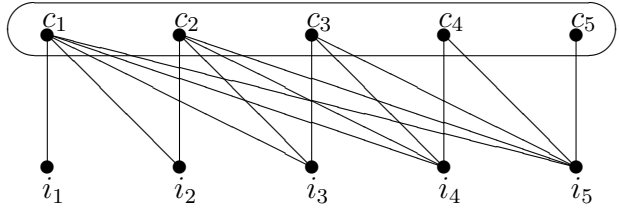

Split graphs of Dilworth number at most 2 are precisely split permutation graphs \cite{BHde85}.
In other words, taking into account (0), we can say that a split graph $G=(C,I,E)$ with a clique $C$ and an independent set $I$
is a permutation graph if and only if
\begin{itemize}
\item[(1)] the vertices of $C$ can be partitioned into at most two sets $C^1$ and $C^2$ so that
both $G[C^1\cup I]$ and $G[C^2\cup I]$ are threshold graphs and  
\item[(2)] the vertices of $I$ can be partitioned into at most two sets $I^1$ and $I^2$ so that
both $G[C\cup I^1]$ and $G[C\cup I^2]$ are threshold graphs.
\end{itemize}

We say that a split permutation graph $G=(C,I,E)$ is {\it symmetric} if it admits 
a partition $C=C^1\cup C^2$ and $I=I^1\cup I^2$ such that 
\begin{itemize}
\item[(3)] both $G[C^1\cup I^1]$ and    
$G[C^2\cup I^2]$ are universal threshold graphs. 
\end{itemize}

\begin{lemma}\label{lem:2n}
Every split permutation graph with $n$ vertices is contained in a symmetric split permutation 
graph with $2n$ vertices as an induced subgraph.
\end{lemma}

\begin{proof}
Let $G=(C,I,E)$ be a split permutation graph with $n$ vertices and with  
a partition $C=C^1\cup C^2$ and $I=I^1\cup I^2$ satisfying (1) and (2). 
For $j=1,2$, we denote $n_j=|C^j|+|I^j|$. 
By adding new vertices to the graph $G[C^1\cup I^1]$ we can extend it to a universal threshold graph with 
$2n_1$ vertices $C_*^1=\{c_1^1,c_2^1,\ldots,c_{n_1}^1\}$, $I_*^1=\{i_1^1,i_2^1,\ldots,i_{n_1}^1\}$ 
so that $N(i_j^1)=\{c_1^1,\ldots,c_j^1\}$ for each $j$. 
Similarly, by adding new vertices to the graph $G[C^2\cup I^2]$ 
we can extend it to a universal threshold graph with 
$2n_2$ vertices $C_*^2=\{c_1^2,c_2^2,\ldots,c_{n_2}^2\}$, $I_*^2=\{i_1^2,i_2^2,\ldots,i_{n_2}^2\}$ 
so that $N(i_j^2)=\{c_1^2,\ldots,c_j^2\}$ for each $j$.

In the graph obtained in this way, condition (3) is satisfied but conditions (1) and (2) are not.
To make the new graph a split permutation graph we add some edges between $I_*^1$ and $C_*^2$ (to make both 
of them vicinal chains) and between $I_*^2$ and $C_*^1$ (also to make these sets vicinal chains).
For $I_*^1$ and $C_*^2$ this can be done as follows: whenever there is an edge $i_k^1c_j^2$ between   
$i_k^1\in I_*^1$ and $c_j^2\in C_*^2$, we add all missing edges between $i_k^1$ and any vertex $c_s^2\in C_*^2$
with $s<j$, and between $c_j^2$ and any vertex $i_s^1\in I_*^1$ with $s>k$. For $I_*^2$ and $C_*^1$, the addition 
of edges can be done by analogy. Since in the original graph the old vertices of each of $I_*^1$, $C_*^2$,
$I_*^2$ and $C_*^1$ form a vicinal chain, no edge has been added between any two old vertices. Therefore, 
the resulting graph contains $G$ as an induced subgraph. 
\end{proof}

\medskip 
Now we establish a correspondence between symmetric split permutation graphs and $321$-avoiding permutations.
More precisely, we deal with {\it labelled} $321$-avoiding permutations and with {\it labelled} symmetric split permutation graphs. 

By a labelled $321$-avoiding permutation we mean a permutation each element of which is assigned one of the two labels 1 or 2
so that the set of $1$-labelled elements forms an increasing sequence and  the set of $2$-labelled elements forms an increasing sequence.
Notice that each $321$-avoiding permutation $\pi$ admits at least two possible labellings, while the total number of labellings is $2^c$, 
where $c$ is the number of connected components of the permutation graph of $\pi$.   
We denote the set of $1$-labelled elements of a permutation $\pi$ by $\pi^{(1)}$ and the set of $2$-labelled elements by $\pi^{(2)}$.
In other words, we assume that a permutation is labelled if it is given {\it together} with a partition $(\pi^{(1)},\pi^{(2)})$
of its elements into two increasing sequences.

By a labelled symmetric split permutation graph we mean a graph given {\it together} with a partition of
its vertices into four sets $C^1,C^2,I^1,I^2$ satisfying (1), (2) and (3).

We also adapt the notion of containment relation (pattern containment and induced subgraph containment)
to the case of labelled objects (permutations and graphs) in a natural way, i.e. elements of one object belonging to the same set
in the given partition must be mapped to elements belonging to the corresponding set in the partition of the other object.   

\begin{lemma}\label{lem:bijection}
There is a one-to-one correspondence $\phi$ between labelled $321$-avoiding permutations and labelled symmetric split permutation graphs.
Moreover, a $321$-avoiding permutation $\pi_1$ contains a $321$-avoiding permutation $\pi_2$ as a pattern if and only if
$\phi(\pi_1)$ contains $\phi(\pi_2)$ as an induced subgraph.   
\end{lemma}
  
\begin{proof}  
Suppose we have a labelled $321$-avoiding permutation $\pi=p_1\ldots p_n$ together with a partition $(\pi^{(1)},\pi^{(2)})$.  
We associate a graph $G$ with $2n$ vertices $C=\{c_1,\ldots,c_n\}$ and 
$I=\{i_1,\ldots,i_n\}$, where 
\begin{itemize}
\item $C=C^1\cup C^2$ with $C^1=\{c_{p_j}\ :\ p_j\in \pi^{(1)}\}$ and 
$C^2=\{c_{p_j}\ :\ p_j\in \pi^{(2)}\}$,  
\item $I=I^1\cup I^2$ with $I^1=\{i_{p_j}\ :\ p_j\in \pi^{(1)}\}$ and 
$I^2=\{i_{p_j}\ :\ p_j\in \pi^{(2)}\}$.
\end{itemize}
We define the set $C$ to be a clique and the set $I$ an independent set, i.e. $G$ is a split graph. 
The set of edges between $C$ and $I$
is defined by describing the neighbourhoods of the vertices of  $I$ as follows:
\begin{itemize}
\item If $p_j\in \pi^{(1)}$, then $N(i_{p_j})=\{c_1,c_2,\ldots,c_{p_j}\}$,
\item if $p_j\in \pi^{(2)}$, then $N(i_{p_j})=\{c_{p_1},c_{p_2},\ldots,c_{p_j}\}$. 
\end{itemize}
By construction, $I^1$ and $I^2$ form vicinal chains.  
Also, from the description it follows that 
\begin{itemize}
\item for any $p_j,p_k\in \pi^{(1)}$ with $p_j<p_k$ we have $N(c_{p_j})\supset N(c_{p_k})$
and therefore $C^1$ forms a vicinal chain, 
\item for any $p_j,p_k\in \pi^{(2)}$ with $p_j<p_k$ we have $N(c_{p_j})\supset N(c_{p_k})$
and therefore $C^2$ forms a vicinal chain.
\end{itemize}
Therefore, $G$ is a split permutation graph. Also, it is not difficult to see that 
\begin{itemize}
\item $|C^1|=|I^1|$ with $N(i_{p_j})\cap C^1=\{c_{p_1},\ldots,c_{p_j}\}\cap \pi^{(1)}$
for each vertex $i_{p_j}\in I^1$,
\item $|C^2|=|I^2|$ with $N(i_{p_j})\cap C^2=\{c_{p_1},\ldots,c_{p_j}\}\cap \pi^{(2)}$
for each vertex $i_{p_j}\in I^2$.
\end{itemize}
Therefore, $G$ is a symmetric split permutation graph. The sets $C^1,C^2,I^1,I^2$ define a labelling of this graph.

The above procedure defines a mapping $\phi$ from the set of labelled 321-avoiding permutations of length $n$ 
to the set of labelled symmetric split permutation graphs with $2n$ vertices. It is not difficult to see 
that different permutations are mapped to different graphs and hence the mapping is injective. 
Now let us show that this mapping is surjective. 

Let $G=(C,I,E)$ be a symmetric split permutation graph with $2n$ vertices given together with 
a partition $C=C^1\cup C^2$ and $I=I^1\cup I^2$ satisfying (1), (2) and (3). Since $G$ is symmetric,
for any two vertices $i,j\in I^1$ we have $|N(i)|\ne |N(j)|$. Similarly, for any two vertices $i,j\in I^2$ 
we have $|N(i)|\ne |N(j)|$. Moreover, the degree of each vertex of $I=I^1\cup I^2$ is a number between 1 and $n$.
Now, define $\pi^{(1)}:=\{|N(i)|\ :\ i\in I^1\}$ and $\pi^{(2)}:=\{1,2, \ldots, n\} \backslash \pi^{(1)}$ and 
let $\pi$ be the permutation obtained by placing the elements of $\pi^{(2)}$ in the positions  $\{|N(i)|\ :\ i\in I^2\}$ 
in the increasing order, and by placing the elements of $\pi^{(1)}$ in the remaining positions also increasingly. 
It is not hard to see that the labelled permutation $\pi$ with the partition $(\pi^{(1)}, \pi^{(2)})$ is mapped to the graph $G$, 
and thus, $\phi$ is a surjection. 

\medskip 

In conclusion, $\phi$ is a bijection between labelled 321-avoiding permutations of length $n$ and 
labelled symmetric split permutation graphs with $2n$ vertices. Moreover, this bijection not only maps a permutation 
$\pi$ to a graph $G=(C,I,E)$, it also maps bijectively the elements of $\pi$ to $C$ and to $I$ as described in the above procedure. 
This mapping from the elements of the permutation to the vertices of the graph proves the second part of the lemma. 
\end{proof}  

\begin{theorem}
There is a split permutation graph with $4n^3$ vertices containing all split permutation graphs with $n$
vertices as induced subgraphs. 
\end{theorem}

\begin{proof}
First, let us define the notion of concatenation of two permutations of the same length.
Given two permutations $\pi_1$ and $\pi_2$ of length $n$, the concatenation $\pi_1\pi_2$ is 
the permutation of length $2n$ obtained by placing the elements of $\pi_2$ to the right of 
$\pi_1$ and by increasing every element of $\pi_2$ by $n$. Also, by $\pi^n$ we denote the 
concatenation of $n$ copies of $\pi$.

Let $\rho_n$ be the $n$-universal $321$-avoiding permutation of Theorem~\ref{theorem}.
Since the permutation graph of $\rho_n$ is connected (see Figure~\ref{fig1}), there are exactly two ways to label 
the elements of $\rho_n$. We denote by $\rho_{n,1}$ and $\rho_{n,2}$ the two labellings of $\rho_n$.
Then the concatenation $\rho_{n,1}\rho_{n,2}$ is universal for labelled $321$-avoiding permutations
whose permutation graphs are connected. Since every  graph with $n$ vertices contains 
at most $n$ connected components, the permutation $u_n:=(\rho_{n,1}\rho_{n,2})^n$ is universal 
for the set of all labelled $321$-avoiding permutations of length $n$. Clearly, $u_n$ is 
$321$-avoiding and hence the graph $\phi(u_n)$ is a symmetric split permutation graph.

Since $u_n$ contains all labelled $321$-avoiding permutations of length $n$, by Lemma~\ref{lem:bijection}
$\phi(u_n)$ contains all labelled symmetric split permutation
graphs with $2n$ vertices. Therefore, by Lemma~\ref{lem:2n}, $\phi(u_n)$
is an $n$-universal split permutation graph.  
Since the length of $\rho_n$ is $n^2$,
the length of $u_n$ is $2n^3$, and hence the number of vertices of $\phi(u_n)$ is $4n^3$.  
\end{proof}

\section{Further research}
\label{sec:con}

One natural question to ask is whether the constructions of universal permutations and universal 
graphs obtained in this paper are best possible in terms of their sizes. As a matter of fact, 
for $n\geq 2$, the permutation $\rho_n$ constructed in Section~\ref{sec:per} does not seem to be 
a shortest possible proper $n$-universal $321$-avoiding permutation. For example, it seems that we can always 
delete a few rightmost elements in this permutation without violating the property of being $n$-universal. 
For the small values of $n$, the optimal permutations with respect to deleting rightmost elements in $\rho_n$ are as follows:
$$
\begin{array}{l|l}
n & \mbox{shortest universal permutation obtained from } \rho_n \mbox{ by deleting rightmost elements}\\
\hline
1 & 1 \\
\hline
2 &  314\\
\hline
3 &  51729\\
\hline
4 & 6182(10)3(12)4(13)5(14)7\\
\hline
5 & 7192(11)3(13)4(15)5(17)6(19)8(21)(10)\\
\end{array}
$$

This observation leads naturally to the following question: how far can we shorten the permutation
$\rho_n$ (not necessarily from the right side) keeping the property of being $n$-universal.
In particular, asymptotically, can we do better than $n^2$ for the length assuming that we start with $\rho_n$ and 
remove some elements? Of course, a natural direction here is not to limit ourselves to $\rho_n$, 
but instead to find a proper $n$-universal $321$-avoiding permutation of the shortest possible length.

\medskip
To state one more open problem, let us observe that the split permutation graphs are precisely 
the permutation graphs of $(2143,3412)$-avoiding permutations. The trivial idea of constructing 
an $n$-universal permutation for this set by concatenating all permutations of length $n$ from  
the set immediately moves us outside of the set, as a forbidden pattern arises. This partly explains 
the difficulty of constructing {\it proper} universal split permutation graphs. Can the construction
obtained in this paper for split permutation graphs be used to construct universal $(2143,3412)$-avoiding permutations?
 
\medskip
Finally, as we mentioned in the introduction, finding a proper {\it infinite} universal split permutation graph with countably many 
vertices is a challenging research problem.

\section*{Acknowledgments}

The authors are grateful to the anonymous referee for many valuable comments that helped to improve significantly the presentation of our results.


\begin{thebibliography}{10}
\bibitem{Alstrup}
{\sc S. Alstrup} and {\sc T. Rauhe},
Small induced-universal graphs and compact implicit graph representations,
{\it Proceedings of the 43rd Annual IEEE Symposium on Foundations of Computer Science}, 2002 53--62. 

\bibitem{BHde85}
{\sc C. Benzaken, P.L. Hammer} and {\sc D. de Werra},
Split graphs of Dilworth number $2$,
{\it Discrete Math}. 55 (1985) 123--127.

\bibitem{Butler}
{\sc S. Butler}, 
Induced-universal graphs for graphs with bounded maximum degree,
{\it Graphs Combin.} 25 (2009) 461--468. 

\bibitem{Chung}
{\sc F.R.K. Chung}, Universal graphs and induced-universal graphs, {\it J. Graph Theory} 14 (1990) 443--454.

\bibitem{CH77}
{\sc V. Chv\'atal} and {\sc P.L. Hammer},
Aggregation of inequalities in integer programming,  
{\it Ann. of Discrete Math}.  1 (1977) 145--162.

\bibitem{ELO08}
{\sc L. Esperet, A. Labourel} and {\sc P. Ochem,} 
On Induced-Universal Graphs for the Class of Bounded-Degree Graphs, 
{\it Inform. Process. Lett}.  108 (2005) 255--260.

\bibitem{FK97}
{\sc Z. F\"uredi} and {\sc P. Komj\'ath}, 
Nonexistence of universal graphs without some trees,
{\it Combinatorica} 17 (1997) 163--171. 

\bibitem{threshold}
{\sc P.L. Hammer} and {\sc A.K. Kelmans}, On universal threshold graphs,
{\it Combin. Probab. Comput.} 3 (1994) 327--344.

\bibitem{implicit}
{\sc S. Kannan, M. Naor} and {\sc S. Rudich}, Implicit representation of graphs, 
{\it SIAM J. Discrete Math}. 5(4) (1992) 596--603.

\bibitem{K} 
{\sc S. Kitaev}, Patterns in permutations and words, Springer-Verlag, 2011.

\bibitem{uni1}
{\sc P. Komj\'ath, A.H. Mekler} and {\sc J. Pach}, 
Some universal graphs, 
{\it Israel J. Math.} 64 (1988) 158--168.

\bibitem{uni2}
{\sc P. Komj\'ath} and {\sc J. Pach}, 
Universal graphs without large bipartite graphs, 
{\it Mathematika} 31 (1984), 282--290.

\bibitem{Lozin}
{\sc V.V. Lozin}, On minimal universal graphs for hereditary classes, (Russian) {\it Diskret. Mat}. 9 (1997), no. 2, 106--115; 
translation in {\it Discrete Math. Appl}. 7 (1997), no. 3, 295--304.

\bibitem{Gabor}
{\sc V.V. Lozin} and {\sc G. Rudolf}, Minimal universal bipartite graphs, {\it Ars Combin}. 84 (2007) 345--356. 

\bibitem{Moon}
{\sc J.W. Moon}, On minimal $n$-universal graphs, {\it  Proc. Glasgow Math. Soc}. 7 (1965)  32--33.

\bibitem{Rado}
{\sc R. Rado}, Universal graphs and universal functions, 
{\it Acta Arith}. 9 (1964) 331--340.
 
\end{thebibliography}
\end{document}